\documentclass[11pt,reqno]{amsart}
\usepackage{amsrefs}
\usepackage{amsfonts, amsmath, amssymb, amscd, amsthm, bm}
\usepackage{cases}
\usepackage{enumerate}

\usepackage{comment}

\usepackage{url}
\usepackage[linktocpage=true,colorlinks,citecolor=magenta,linkcolor=blue,urlcolor=magenta]{hyperref}
\usepackage{multicol}

\newtheorem{thm}{Theorem}[section]
\newtheorem{cor}[thm]{Corollary}

\newtheorem{lem}[thm]{Lemma}
\newtheorem{prop}[thm]{Proposition}
\theoremstyle{definition}

\newtheorem{rem}[thm]{Remark}

\numberwithin{equation}{section}
\def\la{\lambda}
\def\l{\langle}
\def\r{\rangle}
\def\pl{\Delta_p}
\DeclareMathOperator\vol{vol}
\DeclareMathOperator\dive{div}
\DeclareMathOperator\tr{tr}

\DeclareMathOperator\arsinh{arsinh}
\DeclareMathOperator\arcosh{arcosh}
\DeclareMathOperator\Ind{Ind}
\DeclareMathOperator\sd{sh_\delta}
\DeclareMathOperator\cd{ch_\delta}
\DeclareMathOperator\Th{th_\delta}
\DeclareMathOperator\sn{sn_\delta}
\DeclareMathOperator\Tho{th}

\begin{document}
\title{Reilly-type inequalities for submanifolds in Cartan-Hadamard manifolds}

\author{Hang Chen}
\address[Hang Chen]{School of Mathematics and Statistics, Northwestern Polytechnical University, Xi' an 710129, P. R. China \\ email: chenhang86@nwpu.edu.cn}
\thanks{Chen supported by Shaanxi Fundamental Science Research Project for Mathematics and Physics Grant No.~22JSQ005 and Natural Science Foundation of Shaanxi Province Grant No.~2020JQ-101.}
\author{Xudong Gui}
\address[Xudong Gui]{School of Mathematical Sciences, University of Science and Technology of China, Hefei 230026, P. R. China \\ email: guixd7b@mail.ustc.edu.cn}

\begin{abstract}
	Let $M$ be an $m (\ge2)$-dimensional closed orientable submanifold in an $n$-dimensional complete simply-connected Riemannian manifold $N$, where the sectional curvature of $N$ is bounded above by $\delta$. When $\delta<0$, inspired by Niu-Xu (arXiv:2106.01912), we give new upper bounds for the first nonzero eigenvalues of the $p$-Laplacian and the $L_T$ operator, respectively. These generalize Niu-Xu's work for the Laplacian (arXiv:2106.01912)  and improve the estimates due to Chen (Nonlinear Anal.,196, 111833, 2020) for the $p$-Laplacian and Grosjean (Hokkaido Math. J., 33(2), 319--339, 2004) for the $L_T$ operator, respectively. We also obtain several Reilly-type inequalities for the weighted manifolds and some boundary value problems.
\end{abstract}



\maketitle

\section{Introduction}

In geometry of submanifolds,  Reilly-type inequality is an important estimate for the first non-zero eigenvalue of Laplacian, which gives the upper bound in terms of the mean curvature. It originated from Reilly's work \cite{Rei77} in 1977, where closed (orientable) submanifolds in the Euclidean space were considered and a special case says that
\begin{equation}\label{Rei0}
	\lambda_1^{\Delta}\leq\frac{m}{\vol(M)}\int_M|\mathbf{H}|^2.
\end{equation}
Here $m$ and $\mathbf{H}$ are the dimension and the mean curvature vector of the submanifold respectively, and $\lambda_1^{\Delta}$ represents the first non-zero eigenvalue of the Laplacian on the submanifold.
Later, Heintze \cite{Hei88} considered a more general case that the sectional curvature of the ambient manifold $N$ is bounded above by a constant  $\delta$, and showed that
	\begin{numcases}
	{\lambda_{1}^{\Delta}\leq} \frac{m}{\vol(M)}\int_M\big(\delta+|\mathbf{H}|^2\big)& for $\delta\geq 0$, \label{pos}\\
	m\big(\delta+\max |\mathbf{H}|^2\big) &for $\delta< 0$, \label{neg}
\end{numcases}
here it is required that $N$ is simply-connected when $\delta\le 0$ and $M$ lies in a convex ball of $N$ of radius $\le \pi/4\sqrt{\delta}$ when  $\delta>0$. Notice that the approach of Heintze cannot give the upper bound involving $L^2$-norm of the mean curvature when $\delta<0$ while there is a restriction on the radius of the convex ball when $\delta>0$.
However, when $N$ is the space form of constant curvature $\delta$, El Soufi and Ilias \cite{EI92}  removed the restrictions and gave a sharp upper bound
\begin{equation}\label{Rei_c}
	\lambda_1^{\Delta}\leq\frac{m}{\vol(M)}\int_M(\delta+|\mathbf{H}|^2)
\end{equation}
by using a conformal map method whenever $\delta>0$ and $\delta<0$.

Very recently, Niu-Xu did nice work which says the estimate \eqref{neg} can be improved to
\begin{equation}\label{neg1}\\
	{\lambda_{1}^{\Delta}\leq} \frac{m}{\vol(M)}\int_M\big(\delta+|\mathbf{H}|^2\big)
\end{equation}
when $\delta< 0$ and $m\ge 2$ (\cite[Theorem 1.1]{NX21}).

Note that the original Reilly's inequalities in \cite{Rei77} actually involve not only the mean curvature, but also the higher order mean curvatures.
More generally, given a symmetric $(1,1)$-tensor $T=(T_{ij})$ on $M$,
we can define a normal vector field $H_T$ associated with $T$ by
\begin{equation*}
	H_T=\sum_{i=1}^mA(Te_i,e_i)=\sum_{\substack{m+1\leq \alpha\leq n\\ 1\leq i, j\leq m}}h_{ij}^{\alpha}T_{ij}e_{\alpha},
\end{equation*}
where $\{e_1,\cdots, e_n\}$ is a local orthonormal frame of $N$ such that $\{e_1,\cdots, e_m\}$ are tangent to $M$ and $\{e_{m+1},\cdots, e_{n}\}$ are normal to $M$,
and $A=(h_{ij}^\alpha)$ the second fundamental form of $M$ in $N$.
Actually, $H_T$ can be viewed as a natural generalization of the mean curvature and the higher order mean curvature, and has been involved in a lot of results (\cites{Gro04, CL07, CW19a}).
We call $M$ is $T$-minimal if $H_T$ vanishes.

To generalize the Reilly-type inequalities, a common approach is to consider the other operators instead of the Laplacian.
Usually, there are two classes of the elliptic operators as alternatives: one is the $p$-Laplacian and the other one is the $L_T$ operator.
We states them as follows respectively.

\subsection{$p$-Laplacian}
The $p$-Laplacian ($p>1$) on a compact $M$ is defined by
\begin{equation*}
	\pl u=\dive (|\nabla u|^{p-2}\nabla u).
\end{equation*}
It is a second order quasilinear elliptic operator for
$p\neq 2$ and can be viewed as a generalization of the Laplacian (corresponding to $p=2$).
A real number $\la$ is called an eigenvalue if there exists a non-zero function $u$ satisfying the following eigenvalue equation
\begin{equation*}
	\pl u=-\la |u|^{p-2} u\quad  \text{on $M$}
\end{equation*}
with appropriate boundary conditions.
Although the regularity theory of the $p$-Laplacian is very different from the usual Laplacian, there is the smallest positive eigenvalue (i.e., the first non-zero eigenvalue) (cf. \cites{GP87,Le06}), denoted by $\lambda_{1,p}$. When $p=2$, $\lambda_{1,2}$ is just the first non-zero eigenvalue $\lambda_1^{\Delta}$ of the Laplacian.
Moreover, $\lambda_{1,p}$ has a Rayleigh-type variational characterization.
For instance, if $M$ is closed, then (cf. \cite{Ver91})
\begin{equation}\label{Ray}
	\la_{1,p}=\inf\left\{ \left.\frac{\int_M |\nabla u|^p}{\int_M|u|^p}\, \right|\, u\in W^{1,p} (M)\backslash \{0\}, \int_M|u|^{p-2} u =0 \right\}.
\end{equation}

Many estimates for the first non-zero eigenvalue of Laplacian in Riemannian geometry have been generalized to $p$-Laplacian (e.g., \cite{Che20} and the references therein).
In particular, the inequalities mentioned above have been extended from Laplacian to $p$-Laplacian.
For example, Du-Mao \cite{DM15} generalized \eqref{Rei0} in the Euclidean space case; Wei and the first author \cite{CW19} generalized \eqref{Rei_c} by the conformal map method in the space form case.

Recently, the first author generalized \eqref{pos} and \eqref{neg} and proved
\begin{thm}[{cf. \cite[Theorem 1.3]{Che20}}]\label{thm1.1}
	Let $M$ be an $m(\geq 2)$-dimensional  closed orientable submanifold in an $n$-dimensional  Riemannian manifold $N$ of sectional curvature $K_N\leq \delta$. If $\delta\leq 0$, we assume that $N$ is simply-connected;  if $\delta>0$,  we assume that $M$ is contained in a convex ball of radius $\leq \pi/4\sqrt{\delta}$; if $\delta<0$, we assume that $M$ is contained in a convex ball of radius $\leq \frac{1}{2\sqrt{-\delta}}\arsinh(\sqrt{-\delta})$.
	Let $T$ be a  symmetric, positive definite and divergence-free $(1,1)$-tensor on $M$.

	Then the first non-zero eigenvalue  $\lambda_{1,p}$ of the $p$-Laplacian ($p>1$) on $M$ satisfies:

(1) For $\delta=0$,
\begin{equation*}
	\lambda_{1,p}\Big(\int_M\tr T\Big)^{p}\leq n^{\frac{|p-2|}{2}}m^{p/2}\vol(M)\Big(\int_M|H_T|^{\frac{p}{p-1}}\Big)^{p-1}, \quad \mbox{ for } p>1.
\end{equation*}

In particular, when $T=I$ the identity operator, we have
\begin{equation*}
	\lambda_{1,p}\leq n^{\frac{|p-2|}{2}}\frac{m^{p/2}}{(\vol(M))^{p-1}}\Big(\int_M|\mathbf{H}|^{\frac{p}{p-1}}\Big)^{p-1}, \quad \mbox{ for } p>1.
\end{equation*}

	(2) For $\delta>0$,
	\begin{equation*}
		\la_{1,p}\leq (n+1)^{\frac{|p-2|}{2}}m^{p/2}\delta^{p/2}\left(\frac{1}{\vol(M)}\int_M(1+|\mathbf{H}/\sqrt{\delta}|^2)\right)^{\max\{1,p/2\}}.
	\end{equation*}

	(3) For $\delta<0$,
	\begin{equation*}
		\lambda_{1,p}\le n^{\frac{|p-2|}{2}}m\big(\delta+\max|\mathbf{H}|^2\big) \quad \mbox{for } 1<p\le 2.
	\end{equation*}
\end{thm}

The estimate is not good enough in the case of $\delta<0$ for at least two reasons: (i) there is a restriction on the radius of the convex ball for $1<p<2$, but it is not necessary for $p=2$ (cf. Heintze's upper bound \eqref{neg}); (ii) the estimate is just valid for $1<p<2$ and no information is given for $p>2$.

Inspired by Niu-Xu \cite{NX21}, we extend \eqref{neg1} to $p$-Laplacian and give a new upper bound of $\lambda_{1,p}$ without previous restrictions for $\delta<0$. Precisely, we prove the following theorem.
\begin{thm}\label{thm1}
	Let $M$ be an $m(\ge 2)$-dimensional
	closed orientable submanifold in an $n$-dimensional complete simply-connected Riemannian manifold $N$ of sectional curvature $K_N\le  \delta < 0$.
	Then the first non-zero eigenvalue $\la_{1,p}$ of $p$-Laplacian on $M$ satisfies
	\begin{numcases}
		{\la_{1,p}\le  }
		n^\frac{|p-2|}{2}  \Big(\frac{m}{\vol(M)}\Big)^{p/2}\Big( \int_M |\mathbf{H}|^2 \Big)^{p/2-1} \int_M (\delta + |\mathbf{H}|^2),
		& for $p\ge 2$;\label{eq-case1}\\
		n^\frac{|p-2|}{2} m^{p/2} \frac{(-\delta)^{p/2-1}}{\vol(M)} \int_M ( \frac{p}{2}\delta + |\mathbf{H}|^2 ),
		& for $1<p\le2$.\label{eq-case2}
	\end{numcases}
	Moreover, equality implies $p=2$ and $M$ is minimally immersed into a geodesic sphere
	of radius $\arsinh_{\delta}\sqrt{\frac{m}{\lambda_{1}^{\Delta}}}$, where $\arsinh_{\delta}$ is the inverse function of the $\delta$-hyperbolic sine $\sd(y)=\frac{1}{\sqrt{-\delta}}\sinh(\sqrt{-\delta}y)$.
\end{thm}

More generally, we can obtain the upper bound involving $H_T$.

\begin{thm}\label{thm2}
Assumptions as in Theorem \ref{thm1}.
Let $T$ be a symmetric, positive definite and divergence-free $(1,1)$-tensor on $M$.
Suppose the one of the followings holds:

(a)
$(\tr T )I - 2T\ge 0$, i.e., positive semidefinite;

(b) $M$ is contained in a convex ball of radius $\le \frac{1}{2\sqrt{-\delta}}\arcosh\sqrt{2}$.

	Then the first non-zero eigenvalue $\la_{1,p}$ of $p$-Laplacian on $M$ satisfies
	\begin{numcases}
	{\la_{1,p}\le }
	n^\frac{|p-2|}{2} m^{p/2} Q(T)^{p/2-1} \left(\delta+Q(T)\right),
	& for $p\ge 2$;\label{eq-case1'}\\				n^\frac{|p-2|}{2} m^{p/2}(-\delta)^{p/2-1} \left(\frac{p}{2}\delta+Q(T)\right),
	& for $1<p\le2$,\label{eq-case2'}
\end{numcases}
where
\begin{equation*}
	Q(T)=\frac{\vol (M) \sup_M \tr T}{\left(\int_M \tr T\right)^2}\int_M \frac{|\mathbf{H}_T|^2}{\tr T}.
\end{equation*}
Moreover, equality implies $p=2$, $\tr T$ is constant, and $M$ is $T$-minimally immersed into a geodesic sphere of radius $\arsinh_{\delta}\sqrt{\frac{m}{\lambda_{1}^{\Delta}}}$.
\end{thm}

By taking $p=2$, we immediately obtain the following
\begin{cor}
Assumptions as in Theorem \ref{thm2}. We have
\begin{equation}\label{eq-1.10}
	\lambda_{1}^{\Delta}\le m\left( \delta+\frac{\vol (M) \sup_M \tr T}{\left(\int_M \tr T\right)^2}\int_M \frac{|\mathbf{H}_T|^2}{\tr T}\right) .
\end{equation}
\end{cor}

\begin{rem}
(1) The condition $(\tr T) I-2T\ge 0$ is also required in \cite[Theorem 1.1]{CW19a}.

(2) When $m\ge 2$, $(\tr T) I-2T\ge 0$ always holds for $T=I$ the identity operator, and then Theorem \ref{thm2} recovers Theorem \ref{thm1}; furthermore, this recovers \cite[Theorem 1.1]{NX21} when $p=2$.
However, when $m=1$, even for the  Laplacian operator, $(\tr T) I-2T\ge 0$ fails, while an analogous estimate involving the curvature of the curve also fails unless ``smallness'' assumed (see \cite{NX21} for details).
\end{rem}

\subsection{$L_T$ operator}
Consider a symmetric $(1,1)$-tensor $T=(T_{ij})$ on $M$ as previously mentioned. If we additionally assume that $T$ is positive definite and divergence-free, then we can define a self-adjoint elliptic operator $L_T$ associated to $T$ by
\begin{equation*}
	L_Tu=\dive(T\nabla u).
\end{equation*}
The eigenvalue equation is given by $L_Tu=-\lambda u$.
It is also a generalization of the Laplacian since $L_T=\Delta$ when $T=I$.

The $L_T$ operator has been intensively studied. Regarding the Reilly-type inequality, Grosjean proved that
\begin{thm}[\cite{Gro04}*{Theorem 1}]\label{thm-gros}
	Let $M$ be an $m(\geq 2)$-dimensional  closed orientable submanifold in an $n$-dimensional  Riemannian manifold $N$ of sectional curvature $K_N\leq \delta$. If $\delta\leq 0$, we assume that $N$ is simply-connected;  if $\delta>0$,  we assume that $M$ is contained in a convex ball of radius $\leq \pi/4\sqrt{\delta}$.
	Let $T$ be a symmetric, positive definite and divergence-free $(1,1)$-tensor on $M$.
	Then the first non-zero eigenvalue $\la_{1,T}$ of $L_T$ on $M$ satisfies
	\begin{equation}\label{eq-gros}
		\la_{1,T}\le \frac{\sup_M|H_T|^2+\sup_M \delta(\tr T)^2}{\inf_M\tr T}.
	\end{equation}

	If the equality holds, then $M$ is contained in a geodesic sphere.
\end{thm}

When $N$ is the space form with constant curvature $\delta$, the first author and Wang \cite{CW19a} showed that the upper bound in \eqref{eq-gros} can be improved by the $L_2$-norm of $H_T$, precisely,
	\begin{equation}\label{eq-CW}
		\la_{1,T}\le \frac{1}{\vol(M)}\int_M\Big(\delta \tr T+\frac{|H_T|^2}{\tr T}\Big).
	\end{equation}
We remark that in this case, there is no restriction on the radius of $M$, but it is required that $(\tr T )I - 2T\ge 0$, i.e., positive semidefinite. The sufficient  and necessary conditions for the equality holding are also given.

Motivated by Niu-Xu \cite{NX21} again, we can also improve the estimate \eqref{eq-gros} for $\delta<0$. Precisely, we prove the following theorem.
\begin{thm}\label{thm-LT}
	Let $M$ be an $m(\ge 2)$-dimensional closed orientable submanifold in an $n$-dimensional complete simply-connected Riemannian manifold $N$ of sectional curvature $K_N\le  \delta < 0$.
Let $T$ be a symmetric, positive definite and divergence-free $(1,1)$-tensor on $M$.
Suppose the one of the followings holds:

(a)
$(\tr T )I - 2T\ge 0$, i.e., positive semidefinite;

(b) $M$ is contained in a convex ball of radius $\le \frac{1}{2\sqrt{-\delta}}\arcosh\sqrt{2}$.

Then we have
\begin{equation}\label{eq-LT}
	\la_{1,T}\le  \frac{\sup_M\tr T}{\int_M\tr T}\int_M\Big(\delta \tr T+\frac{|H_T|^2}{\tr T}\Big).
\end{equation}
Moreover, equality implies $\tr T$ is constant and $M$ is $T$-minimally immersed into a geodesic sphere of radius $\arsinh_{\delta}\sqrt{\frac{\tr T}{\lambda_{1,T}}}$.
\end{thm}

\smallskip

There are other Reilly-type inequalities in various situations, for example, when the ambient Riemannian manifolds $(N, g_N)$ is endowed with a density $e^{-f}$ for certain $f\in C^\infty(N)$, one can consider the weighted  $L_T$ operator, denoted by $L_{T, f}$; when the submanifold has a non-empty boundary, one can consider the first (non-zero) eigenvalues of the Steklov-type problems (see \cites{Wan09,  IM11, Rot16, CS18, Rot20, DMZ21, MRU21}).
We also obtain some new estimates and list them in the last section.

This paper is organized as follows. In Sect.~2, we list some notations and show some lemmas and propositions, some of which are extensions from the mean curvature to $H_T$.
In Sect.~3 and Sect.~4, we give the proofs for the $p$-Laplacian and $L_T$, respectively.
To estimate the upper bound of the first (non-zero) eigenvalues, a starting point is the Rayleigh-type variational characterization (cf. Eqs. \eqref{Ray}  and \eqref{Ray-T}).
One of key ingredients is to find suitable test functions (Sect.~\ref{Sec: test-func}), and the other one is an estimate of $L^{2}$ lower bound of $H_T$, which extends Niu-Xu's estimate for the mean curvature (Sect.~\ref{Sec: L2}).
In Sect.~5, we give some Reilly-type estimates when the ambient manifold is a weighted manifold.

\textbf{Acknowledgment}:
The first author was supported by Shaanxi Fundamental Science Research Project for Mathematics and Physics Grant No.~22JSQ005 and Natural Science Foundation of Shaanxi Province Grant No.~2020JQ-101.

\section{Preliminaries}
In this section, we introduce some notations and review some lemmas.
\subsection{Some estimates for radial vector fields}
Let $M$ be an $m$-dimensional submanifold of an $n$-dimensional Riemannian manifold $N$ whose sectional curvature $K_N\le \delta<0$.
We denote the Levi-Civita connection on $M$ and $N$ by $\nabla$ and $\nabla^N$ respectively.
For any $q\in N$, we consider its normal coordinates $(x_1,\cdots,x_n)$ on $N$ centered at some fixed point $q_0\in N$.

Let $\sd(y)=\frac{1}{\sqrt{-\delta}}\sinh(\sqrt{-\delta}y)$, $\cd(y)=\sd'(y)$ and $\Th(y)=\sd(y)/\cd(y)$.
It is easy to check that $\cd^2+\delta\sd^2=1$ and $1/\cd^2=1+\delta\Th^2$.

Consider the vector field $Z=\phi(r)\nabla^N r$, where $r=r(q)$ is the distance function from $q_0$ to $q$ and $\phi(r)$ be a positive function with $\phi(r) \sim r$ as $r\to 0$.
Then the coordinates of $Z$ in the normal local frame are $\Big(\dfrac{\phi(r)}{r}x_i\Big)_{1\leq i\leq n}$.
By using basic properties of the exponential map and the comparison for Jacobi fields, we can prove the following lemma.
\begin{lem}[cf. \cites{Gro04, NX21}]\label{lem2.1} Let $T$ be a symmetric, positive definite and divergence-free $(1,1)$-tensor on $M$, and $Z^\top$ be the tangential projection of $Z=\phi(r)\nabla^Nr$ to $M$.
Then we have

	(1) $\sum_{i=1}^n\big\l T\nabla\big(\frac{\phi(r)}{r}x_i\big),\nabla\big(\frac{\phi(r)}{r}x_i\big)\big\r\le(\tr T)\frac{\phi^2}{\sd^2}+ \Big((\phi')^2-\frac{\phi^2}{\sd^2}\Big)\l T\nabla r,\nabla r\r$, and equality holds if $N$ has constant sectional curvature $\delta$.

	(2) $\dive_M(TZ^\top)\ge (\tr T) \frac{\phi}{\Th}+\Big(\phi'-\frac{\phi}{\Th}\Big)\l T\nabla r,\nabla r\r+\l Z, H_T\r$.
\end{lem}
\begin{proof}
	When $\phi=\sd$, Items (1) and (2) go back to  Lemma 1 and Lemma 2 in \cite{Gro04} respectively, and the proof of this lemma follows from the same steps with similar calculations in \cite{Gro04}.
	Here we give a quick proof by directly using the results in \cite{Gro04}.

	Since $$\nabla\big(\frac{\phi(r)}{r}x_i\big)=\big(\frac{\phi}{r}\big)'x_i\nabla r+\frac{\phi}{r}\nabla x_i,\quad  \sum_{i=1}^n(x_i)^2=r^2, \quad\sum_{i=1}^nx_i\nabla x_i=r\nabla r,$$ we have
	\begin{align*}
		\sum_{i=1}^n\big\l T\nabla\big(\frac{\phi(r)}{r}x_i\big),\nabla\big(\frac{\phi(r)}{r}x_i\big)\big\r=&\sum_{i=1}^n\frac{\phi^2}{r^2}\l T\nabla x_i,\nabla x_i\r+\Big(\big(\frac{\phi}{r}\big)'^2r^2+2\phi\big(\frac{\phi}{r}\big)'\Big)|\sqrt{T}\nabla r|^2\\
		=&\sum_{i=1}^n\frac{\phi^2}{r^2}\l T\nabla x_i,\nabla x_i\r+\Big((\phi')^2-\frac{\phi^2}{r^2}\Big)|\sqrt{T}\nabla r|^2\\
		=&(\phi')^2|\sqrt{T}\nabla r|^2-\frac{\phi^2}{\sd^2}|\sqrt{T}\nabla r|^2\\
		&+\frac{\phi^2}{\sd^2}\left[\frac{\sd^2}{r^2}\sum_{i=1}^n\l T\nabla x_i,\nabla x_i\r+(1-\frac{\sd^2}{r^2})|\sqrt{T}\nabla r|^2\right]\\
		\le & (\phi')^2|\sqrt{T}\nabla r|^2-\frac{\phi^2}{\sd^2}|\sqrt{T}\nabla r|^2+\frac{\phi^2}{\sd^2}(\tr T),
	\end{align*}
	where we used \cite[Eqs. (5) and (6)]{Gro04} in the last inequality.
	\begin{align*}
		\dive_M(TZ^\top)=&\dive_M\Big(\frac{\phi}{\sd}T(\sd \nabla r)\Big)\\
		=&\frac{\phi}{\sd}\dive_M(T(\sd \nabla r))+\big(
		\frac{\phi}{\sd}\big)'\l \nabla r, T(\sd\nabla r)\r\\
		\ge &  \frac{\phi}{\sd}\Big((\tr T)\cd +\l \sd  \nabla^N r, H_T\r\Big)+\frac{\phi'\sd-\phi\cd}{\sd}\l \nabla r, T\nabla r\r\\
		= &  (\tr T) \frac{\phi}{\Th}+\Big(\phi'-\frac{\phi}{\Th}\Big)\l T\nabla r,\nabla r\r+\l Z, H_T\r,
	\end{align*}
where we used \cite[Lemma 2]{Gro04} in the inequality.
\end{proof}

\begin{rem}
	When $T=I$, Lemma \ref{lem2.1} recovers \cite[Lemma 2.3]{NX21}.
\end{rem}

\subsection{Test functions}\label{Sec: test-func}
The next lemma will provide suitable test functions for the upper bound estimates of the first eigenvalue, and the basic idea can be found in \cite{Cha78, Hei88}.
	\begin{lem}\label{lem2}
		Under the same assumptions as in Theorem \ref{thm1} or Theorem \ref{thm2}, we have for $p>1$, there exists a point $q_0 \in N$ such that
		\begin{equation}
			\int_M \Big| \frac{\Th(r)}{r} x_i \Big|^{p-2} \frac{\Th(r)}{r} x_i\, dv_M = 0 \, (i=1,\cdots,n).
		\end{equation}
		Here $(x_1,\cdots,x_n)$ is the normal coordinate around $q_0$, and $r(q)=d(q_0,q)$ is the geodesic distance from $q_0$ to $q$ on $N$.
	\end{lem}

	\begin{proof}
		We can assume $M$ lies in a convex ball $B\subset N$.
		For each $q\in N$, we define a vector $Y_q\in T_qN$ under the normal coordinate system around $q$ as follows:
		\begin{equation*}
			Y_q:=\Big(\int_M \Big| \frac{\Th(r)}{r} x_i \Big|^{p-2} \frac{\Th(r)}{r} x_1\, dv_M,\cdots,\int_M \Big| \frac{\Th(r)}{r} x_i \Big|^{p-2} \frac{\Th(r)}{r} x_n\, dv_M \Big).
		\end{equation*}
		Hence, we obtain a vector field $Y$ in a neighborhood of $B$ which points at the boundary into the interior of $B$.

		Now suppose $Y(q)\neq 0$ for all $q\in B$, which means $\Ind(Y)=0$.
		Since $Y$ is a continuous vector field on $B$ and $Y(q)$ points inside for all $q\in \partial B$, we have (cf. \cite{Mor29, Got88})
		\begin{equation*}
			\Ind(Y) +\chi(\partial B)=\chi(B),
		\end{equation*}
		so $\Ind(Y)=1-(1+(-1)^{n-1})=(-1)^{n}\neq 0$, a contradiction. The proof is complete.
	\end{proof}

\begin{rem}\label{rem2.2}
	From the proof we see that $q_0\in B$.
	If the radius of $B$ $\le
	\frac{1}{2\sqrt{-\delta}}\arcosh\sqrt{2}$, then $M$ lies in a ball of radius $\le \frac{1}{\sqrt{-\delta}}\arcosh\sqrt{2}$ centered at $q_0$, which implies $\cd^2(r)\le 2$ on $M$.
\end{rem}

\subsection{$L^2$ lower bound of $H_T$}\label{Sec: L2}
Recall the following key estimate used to prove \eqref{neg1}.
\begin{prop}[\cite{NX21}*{Proposition 2.1}]\label{prop1}
	Under the same assumptions as in Theorem \ref{thm1}, we have
	\begin{equation}\label{eq2.3}
		\int_M |\mathbf{H}|^2  \int_M \Th^2(r) \ge  (\vol(M))^2.
	\end{equation}
Moreover, equality in \eqref{eq2.3} implies that $M$ is minimally immersed in a geodesic sphere.
\end{prop}
We can extend Proposition \ref{prop1} to $H_T$ as follows.
\begin{prop}\label{prop2}
	Under the same assumptions as in Theorem \ref{thm2}, we have 
	\begin{equation}\label{eq2.4}
		\int_M \frac{|H_T|^2}{\tr T}  \int_M
		\tr T\Th^2(r) \ge  \Big(\int_M\tr T\Big)^2.
	\end{equation}
Moreover, equality in \eqref{eq2.4} implies $M$ is $T$-minimally immersed into a geodesic sphere.
\end{prop}
\begin{proof}
	By taking $\phi(r)=\Th(r)$ in Lemma \ref{lem2.1}, we have $Z=\Th(r)\nabla^Nr$ and
	\begin{equation}\label{eq2.5}
		\dive_M(TZ^\top)\ge \tr T +\delta|\sqrt{T}Z^\top|^2+\l Z^\perp, H_T\r.
	\end{equation}
	Integrating \eqref{eq2.5} over $M$ and using the divergence theorem, the Cauchy-Schwarz inequality and the H\"{o}lder inequality, we have
	\begin{equation}\label{eq2.6}
		\|\sqrt{\tr T}Z^\bot\|_2 \left\|\frac{H_T}{\sqrt{\tr T}}\right\|_2 \ge \int_M-\l Z^\perp, H_T\r\ge \int_M \tr T + \delta\|\sqrt{T}Z^\top\|_2^2,
	\end{equation}
where $\|Y\|_2:=(\int_M|Y|^2)^{1/2}$ for some vector field  $Y$, and $^\perp$ represents the normal part of the vector field.

We discuss each case.

\textbf{Case (a): $(\tr T) I-2T\ge 0$. } Let $\Lambda_T$ be the largest eigenvalue of $T$, then $\Lambda_T\le \frac{1}{2}\tr T$,
which implies $|\sqrt{T}Z^\top|^2\le \Lambda_T|Z^\top|^2\le \frac{1}{2}\tr T|Z^\top|^2$, and it follows from \eqref{eq2.6} that
\begin{equation}\label{eq2.6'}
		\|\sqrt{\tr T}Z^\bot\|_2 \left\|\frac{H_T}{\sqrt{\tr T}}\right\|_2 \ge \int_M \tr T + \frac{\delta}{2}\|\sqrt{T}Z^\top\|_2^2.
\end{equation}
Denote $X=\sqrt{\tr T}Z$ and note that $\|X^\top\|_2^2=\|X\|_2^2 -\|X^\bot\|_2^2$, then \eqref{eq2.6'} becomes
	\begin{equation}\label{eq2.7}
		\left( \left\|\frac{H_T}{\sqrt{\tr T}}\right\|_2 + \frac{\delta}{2}\|X^\bot\|_2 \right)\|X^\bot\|_2 \geq \int_M \tr T +  \frac{\delta}{2}\|X\|_2^2.
	\end{equation}

Now we consider the quadratic function $\phi_h(x)=(h+ \frac{\delta}{2} x)x$ with $h=\left\|\frac{H_T}{\sqrt{\tr T}}\right\|_2$.
Note that \eqref{eq2.4} is equivalent to
\begin{equation*}
		\phi_h(\|X\|_2)=\left( h +  \frac{\delta}{2}\|X\|_2 \right)\|X\|_2 \geq \int_M \tr T +  \frac{\delta}{2}\|X\|_2^2,
\end{equation*}
hence, to prove the lemma, it is sufficient to show that
\begin{equation}\label{eq2.8}
	\phi_h(\|X\|_2)\ge\phi_h(\|X^\perp\|_2).
\end{equation}
Next we verify \eqref{eq2.8}.
Firstly,  we have
	\begin{equation}\label{eq2.9}
		2\|X\|_2^2 = 2 \int_M \tr T \Th^2 = \frac{1}{-\delta}\int_M 2\tr T \left( 1-\frac{1}{\cd^2} \right) \le \frac{2}{-\delta}\int_M \tr T.
	\end{equation}
Therefore, we obtain that
	\begin{align*}
		&( \|X\|_2 + \|X^\bot\|_2 )\|X^\bot\|_2\\
		\overset{\phantom{(2.9)}}{=} & \|X\|_2\|X^\bot\|_2 + \|X\|_2^2 - \|X^\top\|_2^2
		\leq 2\|X\|_2^2 - \|X^\top\|_2^2\\
		\overset{\eqref{eq2.9}}{\leq}& \frac{2}{-\delta}\int_M \tr T - \|X^\top\|_2^2
		= \frac{2}{-\delta} \left( \int_M \tr T + \frac{\delta}{2}\|X^\top\|_2^2 \right)\\
		\overset{\eqref{eq2.6'}}{\leq}& \frac{2h}{-\delta} \|X^\bot\|_2.
	\end{align*}
Since $\|X^\bot\|_2\neq 0$, we conclude that $ 2\|X^\bot\|_2\le \|X\|_2 + \|X^\bot\|_2\le \frac{2h}{-\delta}$.
Hence, \eqref{eq2.8} holds by using the monotonicity of $\phi_h=(h+\frac{\delta}{2} x)x$.

\textbf{Case (b): $M\subset B(R), R\le \frac{1}{2\sqrt{-\delta}}\arcosh\sqrt{2}$. } We still denote $X=\sqrt{\tr T}Z$, then \eqref{eq2.6} becomes
\begin{equation}\label{eq2.7'}
	\left( \left\|\frac{H_T}{\sqrt{\tr T}}\right\|_2 + \delta\|X^\bot\|_2 \right)\|X^\bot\|_2 \geq \int_M \tr T +  \delta\|X\|_2^2.
\end{equation}Next we can use the analogous argument as in Case (a) for the quadratic function $\psi_h(x)=(h+ \delta x)x$ with $h=\left\|\frac{H_T}{\sqrt{\tr T}}\right\|_2$.

By Remark \ref{rem2.2} we have $\cd^2 \le2$ on $M$, then
\begin{equation}\label{eq2.9'}
	2\|X\|_2^2 = 2 \int_M \tr T \Th^2 = \frac{1}{-\delta}\int_M 2\tr T \left( 1-\frac{1}{\cd^2} \right) \le \frac{1}{-\delta}\int_M \tr T.
\end{equation}
Consequently, $ 2\|X^\bot\|_2\le \|X\|_2 + \|X^\bot\|_2\le \frac{h}{-\delta}$ follows from
\begin{equation*}
	( \|X\|_2 + \|X^\bot\|_2 )\|X^\bot\|_2
	\le
	2\|X\|_2^2 - \|X^\top\|_2^2
	\overset{\eqref{eq2.9'}}{\leq}
	\frac{1}{-\delta} \left( \int_M \tr T + \delta\|X^\top\|_2^2 \right)
	\overset{\eqref{eq2.6}}{\leq} \frac{h}{-\delta} \|X^\bot\|_2.
\end{equation*}
Now we derive that
\begin{equation*}
	\left( \left\|\frac{H_T}{\sqrt{\tr T}}\right\|_2 +  \delta\|X\|_2 \right)\|X\|_2 =\psi_h(\|X\|_2)\ge \psi_h(\|X^\perp\|_2)\overset{\eqref{eq2.7'}}{\geq}\int_M \tr T +  \delta\|X\|_2^2,
\end{equation*}
which is equivalent to \eqref{eq2.4}.

Finally, we check the necessity for equality in \eqref{eq2.4}.
In both cases (a) and (b), we must have $\|X^\bot\|_2=\|X\|_2$, which implies $X\equiv X^\bot$ and $X^\top\equiv 0$ along $M$.
Further, the first equality in \eqref{eq2.6} gives $X=\sqrt{\tr T} Z$ and $\frac{H_T}{\sqrt{\tr T}}$ are parallel, and $|H_T|=k\tr T|Z|$ for certain positive constant $k$.
On the other hand, equality in \eqref{eq2.5} implies $\tr T=|H_T||Z|=k\tr T |Z|^2$.
Hence, $\Th(r)=|Z|$ is constant on $M$, from which we conclude that $M$ lies in a geodesic sphere.
Since $H_T$ is parallel to $Z$ along $M$, we obtain that $M$ is $T$-minimal in this geodesic sphere.
We complete the proof.
\end{proof}

\begin{rem}\label{rem-radius}
	Assume that equality in \eqref{eq2.4} holds. If we additionally assume that $\tr T$ is constant, then  $|H_T|=k \tr T \Th(r)$ is also constant. Hence, from \eqref{eq2.4} we obtain $k=1/\Th^2(r)$ and $|H_T|=\tr T/\Th(r)$.
\end{rem}

\section{Proofs for $p$-Laplacian with $\delta<0$}
In this section, we prove Theorems \ref{thm1} and \ref{thm2}. Firstly, we give the following estimate for $\la_{1,p}$.
\begin{lem}\label{lem4}
	Let $ M $ be an $ m(\ge   2) $-dimensional closed orientable submanifold in an $ n $-dimensional Riemannian manifold $ N $. Then we have, for all $ p>1 $,

	\begin{equation}\label{eq4.0}
		\la_{1,p} \int_M \Th^p(r) \le  n^\frac{|p-2|}{2} m^\frac{p}{2} \int_M \Big( 1 + \delta\Th^2(r) \Big)^{p/2}.
	\end{equation}

	\begin{proof}
		We choose $ \frac{\Th(r)}{r} x_i $ in Lemma \ref{lem2} as the test functions. Recall the Rayleigh-type quotient \eqref{Ray},
		then we obtain
		\begin{equation}\label{eq4.1}
			\la_{1,p} \int_M \Big| \frac{\Th(r)}{r} x_i \Big|^p \le  \int_M \Big| \nabla ( \frac{\Th(r)}{r} x_i ) \Big|^p\quad  \mbox{for each $i$.}
		\end{equation}
		When $1<p\le 2$,
		\begin{equation}\label{eq4.2}
			\Th^p(r) = (\Th^2(r))^{p/2} = \Big( \sum_{i=1}^n \Big| \frac{\Th(r)}{r} x_i \Big|^2 \Big)^{p/2} \le  \sum_{i=1}^n \Big| \frac{\Th(r)}{r} x_i \Big|^p.
		\end{equation}
	By taking $T=I, f(r)=\Th(r)$ in Item (1) of Lemma \ref{lem2.1} and using the  H\"{o}lder inequality, we derive that
		\begin{align}
			n^{1-2/p} \Big( \sum_{i=1}^n \Big| \nabla \big( \frac{\Th(r)}{r} x_i \big) \Big|^p \Big)^\frac{2}{p} &\le  \sum_{i=1}^n \Big| \nabla \Big( \frac{\Th(r)}{r} x_i \Big) \Big|^2\nonumber\\
			& \le  m \Big( \frac{1}{\cd^2}+\delta \frac{\Th^2}{\cd^2} |\nabla r|^2 \Big)\nonumber\\
			& \le  m \Big( 1 + \delta\Th^2(r) \Big),\label{eq4.3}
		\end{align}
	where we used $\delta<0$ and $1=\cd^2+\delta\sd^2$.
		Then \eqref{eq4.0} is from \eqref{eq4.1}, \eqref{eq4.2} and \eqref{eq4.3}.

		When $ p\ge 2$, the H\"{o}lder inequality gives
		\begin{equation*}
			\Th^2(r) = \sum_{i=1}^n \Big| \frac{\Th(r)}{r} x_i \Big|^2 \le  n^{1-2/p} \Big( \sum_{i=1}^n \Big| \frac{\Th(r)}{r} x_i \Big|^p \Big)^{2/p},
		\end{equation*}
		so we have
		\begin{equation}\label{eq4.5}
			\la_{1,p} \int_M \Th^p(r) \le  n^{p/2-1} \Big( \sum_{i=1}^n \la_{1,p} \int_M \Big| \frac{\Th(r)}{r} x_i \Big|^p \Big).
		\end{equation}
		On the other hand, we have
		\begin{equation}\label{eq4.6}
			\sum_{i=1}^n \Big| \nabla ( \frac{\Th(r)}{r} x_i ) \Big|^p \le  \Big( \sum_{i=1}^n \Big| \nabla \Big( \frac{\Th(r)}{r} x_i \Big) \Big|^2 \Big)^{p/2} \le  m^\frac{p}{2} \Big( 1 + \delta\Th^2(r) \Big)^{p/2}.
		\end{equation}
		Hence, \eqref{eq4.0} follows from \eqref{eq4.1}, \eqref{eq4.5} and \eqref{eq4.6}.
	\end{proof}
\end{lem}

Now we are ready to prove our theorems.
Since Theorem \ref{thm1} is a special case of Theorem \ref{thm2}, we only prove the latter.

\begin{proof}[Proof of Theorem \ref{thm2}]
We have $\delta\Th^2(r)=-\Tho^2(\sqrt{-\delta}r)\in (-1,0]$ since $\delta<0$.
From \eqref{eq2.4} we have
\begin{equation}\label{eq3.7}
	\sup_M \tr T \int_M \Th^2\int_M \frac{|H_T|^2}{\tr T}  \geq \int_M \tr T \Th^2 \int_M \frac{|H_T|^2}{\tr T} \geq \left(\int_M \tr T\right)^2.
\end{equation}
	\textbf{Case 1.}  When $ p \ge 2 $, we have
	\begin{align}
		\la_{1,p} \int_M \Th^p(r)
			& \le  n^\frac{|p-2|}{2} m^\frac{p}{2} \int_M \Big( 1 + \delta\Th^2(r) \Big)^{p/2}\nonumber\\
			& \le  n^\frac{|p-2|}{2} m^\frac{p}{2} \int_M \Big(1 + \delta\Th^2(r)\Big) \label{eq-Case1}\\
			& =  n^\frac{|p-2|}{2} m^\frac{p}{2}   \Big(\vol(M)+ \delta\int_M\Th^2(r) \Big)\nonumber\\
			&\le n^\frac{|p-2|}{2} m^\frac{p}{2} \left(\frac{\vol(M) \sup_M \tr T}{\left(\int_M \tr T\right)^2}\int_M \frac{|H_T|^2}{\tr T}+\delta\right) \int_M \Th^2(r),\label{eq3.8}
	\end{align}
where we used  \eqref{eq3.7} in the last inequality.
	By the H\"{o}lder inequality, we have
	\begin{equation*}
		\int_M \Th^2(r) \le  \Big( \int_M \Th^p(r) \Big)^{2/p} (\vol(M))^{1-2/p},
	\end{equation*}
	which is equivalent to
	\begin{equation}\label{eq3.9}
		(\vol(M))^{1-p/2} \Big( \int_M \Th^2(r) \Big)^{p/2} \le  \int_M \Th^p(r).
	\end{equation}
	Combining \eqref{eq3.9} with \eqref{eq3.8}, we have
	\begin{align*}
		\la_{1,p}
		&\le n^\frac{|p-2|}{2} m^\frac{p}{2} \left(\frac{\vol(M) \sup_M \tr T}{\left(\int_M \tr T\right)^2}\int_M \frac{|H_T|^2}{\tr T}+\delta\right)\left( \frac{\vol(M)}{\int_M \Th^2(r)} \right)^{p/2-1}\\
		& \le  n^\frac{|p-2|}{2} m^{p/2} \left(\frac{\vol(M) \sup_M \tr T}{\left(\int_M \tr T\right)^2}\int_M \frac{|H_T|^2}{\tr T}+\delta\right)\left( \frac{\vol(M) \sup_M \tr T}{\left(\int_M \tr T\right)^2}\int_M \frac{|H_T|^2}{\tr T} \right)^{p/2-1},
	\end{align*}
where we used \eqref{eq3.7} again. This is just \eqref{eq-case1'}.

\textbf{Case 2.} When $ 1<p \le  2 $, we have
\begin{align}
	\la_{1,p} \int_M \Th^p(r)
	& \le  n^\frac{|p-2|}{2} m^\frac{p}{2} \int_M \Big( 1 + \delta\Th^2(r) \Big)^{p/2}\nonumber\\
	& \le  n^\frac{|p-2|}{2} m^\frac{p}{2} \int_M \Big( 1 + \frac{p}{2} \delta\Th^2(r) \Big)\label{eq-Case2}\\
	& = n^\frac{|p-2|}{2} m^\frac{p}{2} \left( \vol(M) + \frac{p}{2}\delta \int_M \Th^2(r) \right)\nonumber\\
	& \le  n^\frac{|p-2|}{2} m^\frac{p}{2}  \left( \frac{\vol(M) \sup_M \tr T}{\left(\int_M \tr T\right)^2}\int_M \frac{|H_T|^2}{\tr T} + \frac{p}{2} \delta \right)\int_M\Th^2(r),\label{eq3.10}
\end{align}
where we used  \eqref{eq3.7} in the last inequality.
Notice that
\begin{equation*}
	\Th^p(r) = \Tho^p(\sqrt{-\delta}r) / (-\delta)^\frac{p}{2}
\end{equation*}
and
\begin{equation*}
	\Th^2(r) = \Tho^2(\sqrt{-\delta}r) / (-\delta) \le  \Tho^p(\sqrt{-\delta}r) / (-\delta),
\end{equation*}
so
\begin{equation*}
	\int_M \Th^2(r)\le  (-\delta)^{\frac{p}{2}-1}\int_M\Th^p(r) .
\end{equation*}
Inserting this into \eqref{eq3.10} we derive  \eqref{eq-case2'}.

Finally, we check what happens when equality holds in \eqref{eq-case1'} or \eqref{eq-case2'}. Equality in either \eqref{eq-Case1} or \eqref{eq-Case2} implies $p=2$; the first equality in \eqref{eq3.7} implies that $\tr T$ is constant, while the second equality in \eqref{eq3.7} implies that $M$ is $T$-minimally immersed into a geodesic sphere (cf. Proposition \ref{prop2}), and the radius can be determined by Remark \ref{rem-radius}.
Hence, we complete the whole proof.
\end{proof}

\section{Proof for $L_T$ operators with $\delta<0$}
In this section, we prove Theorem \ref{thm-LT}. Actually, we will prove the following stronger result.
\begin{thm}\label{thm-TS}
	Let $M$ be an $m(\ge 2)$-dimensional closed orientable submanifold in an $n$-dimensional complete simply-connected Riemannian manifold $N$ of sectional curvature $K_N\le  \delta < 0$.
	Let $T, S$ be a symmetric, positive definite and divergence-free $(1,1)$-tensor on $M$.
	Suppose the one of the followings holds:

	(a)
	$(\tr S )I - 2S\ge 0$, i.e., positive semidefinite;

	(b) $M$ is contained in a convex ball of radius $\le \frac{1}{2\sqrt{-\delta}}\arcosh\sqrt{2}$.

	Then we have the following two estimates for the upper bound of the first non-zero eigenvalue $\lambda_{1,T}$ of $L_T$:
	\begin{equation}\label{eq-TS1}
		\la_{1,T}\le \sup_M\Big(\frac{\tr T}{\tr S}\Big)\sup_M (\tr S) \Big(\delta+\frac{1}{\int_M\tr S}\int_M\frac{|H_S|^2}{\tr S}\Big),
	\end{equation}
and \begin{equation}\label{eq-TS2}
		\la_{1,T}\le \sup_M(\tr T) \Big(\delta +\frac{\vol(M) \sup_M\tr S}{(\int_M\tr S)^2}\int_M\frac{|H_S|^2}{\tr S}\Big).
\end{equation}
Moreover, equality in either \eqref{eq-TS1} or \eqref{eq-TS2} implies that both $\tr S$ and $\tr T$ are constant, and that $M$ is $S$-minimally immersed into a geodesic sphere of radius $\arsinh_{\delta}\sqrt{\frac{\tr T}{\lambda_{1,T}}}$.
\end{thm}
\begin{proof}
	We also have the Rayleigh-type quotient (cf. \cite{Gro04})
		\begin{equation}\label{Ray-T}
		\la_{1,T}=\inf\left\{ \left.\frac{-\int_M uL_{T}u}{\int_M u^2}\, \right|\, u\in W^{1,2} (M)\backslash \{0\}, \int_M u =0 \right\}.
	\end{equation}

By slightly modifying Lemma \ref{lem2}, we can prove that there exists a point $q_0 \in N$ such that
\begin{equation}
	\int_M  \frac{\Th(r)}{r} x_i\, dv_M = 0 \, (i=1,\cdots,n),
\end{equation}
where $(x_1,\cdots,x_n)$ is the normal coordinate around $q_0$, and $r(q)=d(q_0,q)$ is the geodesic distance from $q_0$ to $q$ on $N$. Here we omit the detail.

Now by using Lemma \ref{lem2.1}, we have
\begin{align}
	\la_{1,T} \int_M \Th^2(r)&\le -\int_M \Big(\frac{\Th(r)}{r}x_i\Big)L_T\Big(\frac{\Th(r)}{r}x_i\Big)\nonumber\\
	&= \sum_{i=1}^n\int_M \Big\l T\nabla\Big(\frac{\Th(r)}{r}x_i\Big),\nabla\Big(\frac{\Th(r)}{r}x_i\Big)\Big\r\nonumber\\
	& \le   \int_M (\tr T)\Big( 1 + \delta\Th^2(r) \Big). \label{eq 4.4}
\end{align}

In order to introduce $\tr S$ and $H_S$ and cancel the term of $\int_M\Th^2(r)$, we have two methods.

Method 1. \begin{align}
	&\int_M (\tr T)\Big( 1 + \delta\Th^2(r) \Big) \nonumber\\
	\le& \sup_M\Big(\frac{\tr T}{\tr S}\Big) \int_M (\tr S)\Big( 1 + \delta\Th^2(r) \Big)\nonumber\\
	= &\sup_M\Big(\frac{\tr T}{\tr S}\Big)  \Big(\frac{\int_M\tr S}{\int_M \tr S\Th^2(r)} +\delta  \Big)\int_M\tr S\Th^2(r)\nonumber\\
	\le& \sup_M\Big(\frac{\tr T}{\tr S}\Big) \sup_M(\tr S) \Big(
	\frac{1}{\int_M\tr S} \int_M\frac{|H_S|^2}{\tr S}+\delta  \Big)\int_M\Th^2(r),\label{eq-4.5}
\end{align}
where we used Proposition \ref{prop2} for the tensor $S$ in the last inequality.
Putting \eqref{eq-4.5} into \eqref{eq 4.4}, we obtain \eqref{eq-TS1}.

Method 2.  \begin{align}
	&\int_M (\tr T)\Big( 1 + \delta\Th^2(r) \Big) \nonumber\\
	\le& \sup_M(\tr T) \int_M \Big( 1 + \delta\Th^2(r) \Big)\nonumber\\
	= &\sup_M(\tr T) \Big( \vol(M) \times 1+ \delta\int_M \Th^2(r) \Big)\nonumber\\
	\le& \sup_M(\tr T) \Big( \vol(M)\frac{\int_M\tr S \Th^2(r)}{(\int_M\tr S)^2} \int\frac{|H_S|^2}{\tr S}+ \delta\int_M \Th^2(r) \Big)\nonumber\\
	\le& \sup_M(\tr T) \Big( \frac{\vol(M)\sup_M\tr S }{(\int_M\tr S)^2} \int\frac{|H_S|^2}{\tr S}+ \delta \Big)\int_M \Th^2(r),\label{eq4.7}
\end{align}
where we used Proposition \ref{prop2} for the tensor $S$ in the penultimate inequality.
Putting \eqref{eq4.7} into \eqref{eq 4.4}, we obtain \eqref{eq-TS2}.

Finally, the necessity for equality in either \eqref{eq-TS1} or \eqref{eq-TS2} follows from \eqref{eq-4.5} or \eqref{eq4.7} combining with Proposition \ref{prop2} and Remark \ref{rem-radius}.
\end{proof}

\begin{rem}
	We  briefly discuss the upper bounds in Eqs.~\eqref{eq-TS1} and \eqref{eq-TS2}.

	(1) When $S=T$, \eqref{eq-TS1} is better than \eqref{eq-TS2}, and Theorem \ref{thm-TS} reduces to Theorem \ref{thm-LT}.

	(2) When $\tr S$ is constant, \eqref{eq-TS1}  is the same as  \eqref{eq-TS2}.

	(3) Generally, $\sup_M\Big(\frac{\tr T}{\tr S}\Big)\sup_M (\tr S)\ge \sup_M\tr T$, but we cannot  compare $ \sup_M\Big(\frac{\tr T}{\tr S}\Big) $ and $\frac{\vol(M) \sup_M \tr T }{\int_M\tr S}$.
	For example, when $\tr T$ is constant, we have $\sup_M\Big(\frac{\tr T}{\tr S}\Big) \ge \frac{\vol(M) \sup_M \tr T }{\int_M\tr S}$, so we cannot compare these two upper bounds since $\delta<0$.

	(4) When $T=I$, \eqref{eq-TS1} recovers \eqref{eq-1.10}.
\end{rem}

\section{Weighted manifolds and boundary value problems}
In this section, we consider the weighted manifolds and some boundary value problems and give the corresponding Reilly-type inequalities.
Since most of the discussion is similar to the previous section, we will highlight the differences and omit other details.

\subsection{Basic settings and notations}\label{Sec: wm}
Firstly we recall some concepts of weighted manifolds.
Let $(N, \bar{g})$ be a Riemannian manifold equipped with the metric $\bar{g}$.
We denote by $dv_{\bar{g}}$ by the usual (non-weighed) volume form on $N$.
For a smooth function $f$ on $N$, we can consider the \emph{weighted volume form} (or \emph{weighted measure}) $\bar{\mu}_f:=e^{-f}dv_{\bar{g}}$.
The triplet $(N, \bar{g}, \bar{\mu}_f)$ is called a \emph{weighted manifold}.
It is also called a \emph{metric measure space}, or a \emph{manifold with density} in some literature.

Now let $M$ be a submanifold of $(N, \bar{g}, \bar{\mu}_f)$.
Then $M$ has an induced metric $g$ and a weighted measure $\mu_f=e^{-f}dv_g$.
If $M$ has non-empty boundary $\partial M$, then we denote by $\tilde{g}$ and $\tilde{\mu}_f=e^{-f}dv_{\tilde{g}}$ on $\partial M$ the induced metric and the weighted measure  on $\partial M$, respectively.

Denote $\dim M=m,\, \dim N=n$. The corresponding   $n$-volume $\vol_f(N)$, $m$-volume $\vol_f(M)$ and $(m-1)$-volume $\vol_f(\partial M)$ (when $\partial M\neq \emptyset$) are respectively given by
\begin{equation*}
	\vol_f(N)=\int_N\bar{\mu}_f,\quad  \vol_f(M)=\int_M\mu_f, \quad  \vol_f(\partial M)=\int_{\partial M}\mu_f.
\end{equation*}

For a symmetric $(1,1)$-tensor $T=(T_{ij})$ on $M$, we can define the  weighted divergence (or $f$-divergence) on $M$ by
\begin{equation*}
	\dive_{f, M}=\dive_M-\l \nabla f, \cdot\r.
\end{equation*}
Then the weighted $L_T$ operator, denoted by $L_{T,f}$, can be defined by
\begin{equation*}
	L_{T,f}u:=\dive_{f, M}(T\nabla u)=\dive_M(T\nabla u)-\l \nabla f, T\nabla u\r.
\end{equation*}

One can easily obtain the Stokes' formula  for the weighted operators with respect to the weighted measure $\mu_f$.
When $T$ is divergence-free, $L_{T,f}$ is self-adjoint with respect to the weighted measure $\mu_f$.
In particular, the case of $T=I$ has been widely studied, and in this situation, $L_{T, f}$ becomes $\Delta u-\l \nabla f, \nabla u\r$, which is usually called the weighted Laplacian or $f$-Laplacian.

\subsection{The case $\delta=0$}
Let $\sn(t)= \frac{1}{\sqrt{\delta}}\sin(\sqrt{\delta}t), t$ and $\frac{1}{\sqrt{-\delta}}\sinh(\sqrt{-\delta}t)$ when $\delta>0, \delta=0$ and $\delta<0$, respectively.
Very recently, Manfio, Roth and Upadhyay \cite{MRU21} studied various extrinsic upper bounds in weighted manifolds and proved the following three results, which are restated by using the notations and settings in Sect.~\ref{Sec: wm} as follows\footnote{In fact, the estimates in the case $K_N\le \delta$ with $\delta>0$ are also obtained, but here we focus on the case $K_N\le 0$  and assume that $N$ is simply-connected.}.

\begin{thm}[cf. \cite{MRU21}*{Theorem 1.1}]\label{MRU-thm1}
	Let $M$ be a closed submanifold immersed in a weighted manifold $(N, \bar{g},  \bar{\mu}_f)$ with sectional curvature $K_{N}\le \delta\le 0$.
	Let $T, S$ be symmetric, divergence-free and positive definite $(1, 1)$-tensors on $M$.
	Then  the first positive eigenvalue of $L_{T,f}$, denoted by $\la_1$, satisfies
	\begin{equation}\label{MRU-eq1}
	\la_1\le \sup_M\left[\delta \tr T+\sup_M\left(\frac{|H_T-T\nabla f|}{\tr S}\right)|H_S-S\nabla f|\right].
	\end{equation}
\end{thm}

\begin{thm}[cf. \cite{MRU21}*{Theorem 1.2}]\label{MRU-thm2}
Let $M$ be a compact submanifold with non-empty boundary $\partial M$ immersed in a weighted manifold $(N, \bar{g},  \bar{\mu}_f)$ with sectional curvature $K_{N}\le \delta\le 0$.
	Let $T, S$ be symmetric, divergence-free and positive definite $(1, 1)$-tensors on $M$ and $\partial M$, respectively, and denote by $\sigma_1$ the first eigenvalue of the generalized weighted Steklov problem
	\begin{equation*}
		(f\mbox{-}\mathrm{Stek})\,\,\begin{cases}
			L_{T,f} u=0, & \mathrm{in\ } M;\\
			\quad \frac{\partial u}{\partial \nu_T}=\sigma u, & \mathrm{on\ } \partial M,
		\end{cases}
	\end{equation*}
	where $\frac{\partial u}{\partial \nu_T}=\l T(\tilde{\nabla}u),\nu\r$, $\tilde{\nabla}$ and $\nu$ are the Levi-Civita connection and unit normal on $\partial M$, respectively.

	Assume that $M$ is contained in the geodesic ball $B(p, R)$ of radius $R$, where $p$ is the center of mass of $\partial M$ for the measure $\tilde{\mu}_f$, then
	\begin{align}\label{MRU-eq2}
		\sigma_1\le &\sup_M\left[\delta \tr T+\sup_M\left(\frac{|H_T-T\nabla f|}{\tr T}\right)|H_T-T\nabla f|\right]\nonumber\\
		&\times \left[\delta+\frac{\sup_{\partial M}|H_S-S(\tilde{\nabla}f)|^2}{\inf_{\partial M}(\tr S)^2}\right]\frac{\vol_f(M)}{\vol_f(\partial M)}\sn^2(R).
	\end{align}
\end{thm}

\begin{thm}[cf. \cite{MRU21}*{Theorem 1.3}]\label{MRU-thm3}
Let $M$ be a compact submanifold with non-empty boundary $\partial M$ immersed in a Riemannian manifold $(N, \bar{g})$ with sectional curvature $K_{N}\le \delta\le 0$.
	Let $S$ be a symmetric, divergence-free and positive definite $(1, 1)$-tensors on  $\partial M$, and denote by $\alpha_1$ the first non-zero eigenvalue of the Steklov-Wentzell problem
	\begin{equation*}
		\mathrm{(SW)}\,\,\begin{cases}
			\qquad\quad\,\,\,\Delta u=0, & \mathrm{in\ } M;\\
			-b \tilde{\Delta}u-\frac{\partial u}{\partial \nu}=\alpha u, & \mathrm{on\ } \partial M,
		\end{cases}
	\end{equation*}
	where $b$ is a given positive constant and $\tilde{\Delta}$ is the Laplacian on $\partial M$.

	Assume that $M$ is contained in the geodesic ball $B(p, R)$ of radius $R$, where $p$ is the center of mass of $\partial M$, then
	\begin{align}\label{MRU-eq3}
		\alpha_1\le &\left[m\frac{\vol(M)}{\vol(\partial M)}+b(m-1)-\delta \sn^2(R)\left(\frac{\vol(M)}{\vol(\partial M)}+b\right)\right]\nonumber\\
		&\times \left(\delta+\frac{\sup_{\partial M}|H_S|^2}{\inf_{\partial M}(\tr S)^2}\right).
	\end{align}
\end{thm}

When $\delta=0$, we can obtain upper bounds in terms of the integrals. Precisely, we have
\begin{thm}\label{prop-1}
	Settings and notations as in Theorem \ref{MRU-thm1}  with $\delta=0$. Then
	\begin{equation}\label{eq-prop-1}
		\la_1\Big(\int_M \tr S \, \mu_f\Big)^2\le \int_M\tr T\, \mu_f\int_M|H_S-S\nabla f|^2\,\mu_f.
	\end{equation}
\end{thm}

\begin{thm}\label{prop-2}
	Settings and notations as in Theorem \ref{MRU-thm2} with $\delta=0$. Then
	\begin{equation}\label{eq-prop-2}
		\sigma_1\Big(\int_{\partial M} \tr S \, \tilde{\mu}_f\Big)^2\le \int_M\tr T\, \mu_f\int_{\partial M}|H_S-S\nabla f|^2\,\tilde{\mu}_f.
	\end{equation}
\end{thm}

\begin{thm}\label{prop-3}
	Settings and notations as in Theorem \ref{MRU-thm3}  with $\delta=0$. Then
	\begin{equation}\label{eq-prop-3}
		\alpha_1\Big(\int_{\partial M} \tr S \Big)^2\le \Big(m\vol(M)+b(m-1)\vol(\partial M)\Big)\int_{\partial  M}|H_S|^2.
	\end{equation}
\end{thm}
\begin{proof}[Proofs of Theorems \ref{prop-1}, \ref{prop-2} and \ref{prop-3}]
	We have the variational characterizations (cf. \cites{DKL16, Auc04, IM11, Rot20, MRU21}):
	\begin{gather}\label{Ray-LTf}
		\la_1=\inf
		\left\{ \left.\frac{\int_M \l T\nabla u, \nabla u\r\mu_f}{\int_M u^2\,\mu_f}\right|\,  \int_{ M} u\,\mu_f =0 \right\},\\
		\label{Ray-Stek}
		\sigma_1=\inf
		\left\{ \left.\frac{\int_M \l T\nabla u, \nabla u\r\,\mu_f}{\int_{\partial M} u^2\,\tilde{\mu}_f}\, \right|\,  \int_{\partial M} u\, \tilde{\mu}_f =0 \right\},\\
		\label{Ray-WL}
		\alpha_1=\inf
		\left\{ \left.\frac{\int_M |\nabla u|^2 +b\int_{\partial M} |\tilde{\nabla} u|^2}{\int_M u^2}\, \right|\,  \int_{\partial M} u =0 \right\}.
	\end{gather}

	When $\delta=0$, consider $Z=r\nabla^N r$, we have (cf. \cite[Lemma 2.1]{MRU21})
	\begin{equation}
		\dive_{f, M}(TZ^\top)\ge \tr T+\l Z, H_T-T\nabla f\r,
		\label{eq5.17}
	\end{equation}
which implies
\begin{equation}
	\int_M\tr T\,\mu_f
	\le -\int_M\l Z, H_T-T\nabla f \r\,\mu_f\le \int_M |H_T-T\nabla f|r\,\mu_f\label{eq5.19}
\end{equation}
if $M$ has no boundary.

Similar to the proof of Lemma \ref{lem2}, there exists a point $q_0 \in N$ such that
\begin{equation}\label{eq5.21}
	\int_M   x_i\, \mu_f= 0 \, (i=1,\cdots,n),
\end{equation}
or
\begin{equation}\label{eq5.22}
	\int_{\partial M}   x_i\, \tilde{\mu}_f = 0 \, (i=1,\cdots,n),
\end{equation}

By taking $\phi(r)=r$ in Item (1) of Lemma \ref{lem2.1},  we have
$\sum_{i=1}^n\big\l T\nabla x_i,\nabla x_i\big\r\le\tr T$. In particular, $\sum_{i=1}^n\big\l \nabla x_i,\nabla x_i\big\r\le m$ for $T=I$. We also have $\sum_{i=1}^n\big\l \tilde{\nabla} x_i,\tilde{\nabla} x_i\big\r\le m-1$ when we consider the boundary $\partial M$.

We use \eqref{Ray-LTf}, \eqref{Ray-Stek} and \eqref{Ray-WL} respectively.
Considering \eqref{eq5.21}, we have
\begin{align*}
	\la_1\int_M r^2\,\mu_f=\la_1\int_M \sum_{i=1}^n(x_i)^2\,\mu_f\le\int_M \sum_{i=1}^n\big\l T\nabla x_i,\nabla x_i\big\r\,\mu_f\le\int_M\tr T\,\mu_f.
\end{align*}
Multiplying the both sides by $\int_M |H_S-S\nabla f |^2\,\mu_f$, we derive that
\begin{align*}
	\Big(\int_M |H_S-S\nabla f |^2\,\mu_f\Big)\Big(\int_M\tr T\,\mu_f\Big)\ge &\la_1\Big(\int_M r^2\,\mu_f\Big)\Big(\int_M |H_S-S\nabla f |^2\,\mu_f\Big) \\
	\ge & \la_1\Big(\int_M r|H_S-S\nabla f |\,\mu_f\Big)^2\\
	\ge& \la_1 \Big(\int_M\tr S\,\mu_f\Big)^2,
\end{align*}
where we used \eqref{eq5.19} for the tensor $S$ and the H\"{o}lder inequality. This is \eqref{eq-prop-1}.

Considering \eqref{eq5.22}, we have
\begin{align*}
	\sigma_1\int_{\partial M} r^2\,\tilde{\mu}_f\le\int_M \sum_{i=1}^n\big\l T\nabla x_i,\nabla x_i\big\r\,\mu_f\le\int_M\tr T\,\mu_f.
\end{align*}
Multiplying the both sides by $\int_{\partial M} |H_S-S\nabla f |^2\,\tilde{\mu}_f$, we can obtain \eqref{eq-prop-2} by using the similar arguments.

Considering \eqref{eq5.22} for $f\equiv 0$, we have
\begin{align*}
	\alpha_1\int_{\partial M} r^2 \le \int_M \sum_{i=1}^n\big\l \nabla x_i,\nabla x_i\big\r+b\int_{\partial M} \sum_{i=1}^n\big\l \tilde{\nabla} x_i,\tilde{\nabla} x_i\big\r\le m\vol(M)+b(m-1)\vol(\partial M).
\end{align*}
Multiplying the both sides by $\int_{\partial M} |H_S|^2$, we can obtain \eqref{eq-prop-3} by using the similar arguments.
\end{proof}

\begin{rem}
	Our upper bounds are formally the same as the previous results such as in \cite{Rot16, CS18,Rot20}. The main differences are:

	(1) In \cite{Rot16, CS18,Rot20},  the ambient manifold is the Euclidean space. Here we extend it to a Cartan-Hadamard manifold.

	(2) When the ambient manifold is   the Euclidean space, one can choose test functions by moving the mass center to the origin and use the Hsiung–Minkowski formulas; when the ambient manifold is a  Cartan-Hadamard manifold, we choose test functions by (a modified version of) Lemma \ref{lem2}, and use the inequality \eqref{eq5.17}.

	(3) When the ambient space is a Cartan-Hadamard manifold, giving the sufficient and necessary conditions for the equalities become more difficult.
\end{rem}


\end{document}